\documentclass[11pt]{article}

\usepackage{amsmath}
\usepackage{amssymb}
\usepackage{amsthm}
\usepackage{bbm}
\usepackage{graphicx}
\usepackage{color}
\usepackage{epsfig}

\theoremstyle{definition}

\newtheorem{theorem}{Theorem}

\newtheorem{prop}[theorem]{Proposition}
\newtheorem{lemma}[theorem]{Lemma}
\newtheorem{cor}[theorem]{Corollary}
\newtheorem{definition}[theorem]{Definition}
\newtheorem{exm}[theorem]{Example}
\newtheorem{alg}[theorem]{Algorithm}

\newtheorem{claim*}{Claim}

\newcommand{\odd}{\ensuremath{\textsc{odd}}}
\newcommand{\even}{\ensuremath{\textsc{even}}}

\newcommand{\floor}[1]{\ensuremath{\left \lfloor {#1} \right \rfloor}}
\newcommand{\ceil}[1]{\ensuremath{\left \lceil {#1} \right \rceil}}

\newcommand{\N}{{\mathbb N}}
\newcommand{\Z}{{\mathbb Z}}

\newcommand{\R}{{\mathbb R}}

\newcommand{\cL}{{\mathcal L}}

\newcommand{\dotcup}{\ensuremath{\mathaccent\cdot\cup}}

\allowdisplaybreaks

\begin{document}

\title{Linearly bounded liars, adaptive covering codes, and
deterministic random walks}
\author{%
Joshua N.\ Cooper \\
 \small University of South Caroloina,  Columbia, South Carolina\\
 \small{\tt cooper@math.sc.edu}
\and
Robert B.\ Ellis\thanks{Project sponsored by the National
Security Agency under Grant Number
\#H98230-07-1-0029. The United States Government is authorized
to reproduce and distribute reprints notwithstanding any copyright
notation herein.} \\
 \small Illinois Institute of Technology, Chicago, IL\\
 \small{\tt rellis@math.iit.edu}
}

\date{\today}
\maketitle

\begin{abstract}
We analyze a deterministic form of the random walk on the
integer line called the {\em liar machine}, similar to the
rotor-router model, finding asymptotically tight pointwise and interval discrepancy
bounds versus random walk. This provides an improvement in
the best-known winning strategies in the binary symmetric
pathological liar game with a linear fraction of responses
allowed to be lies. Equivalently, this proves the existence of
adaptive binary block covering codes with block length $n$,
covering radius $\leq fn$ for $f\in(0,1/2)$, and cardinality
$O(\sqrt{\log \log n}/(1-2f))$ times the sphere bound
$2^n/\binom{n}{\leq \lfloor fn\rfloor}$.
\end{abstract}

\section{Introduction}

In this paper we employ machinery of deterministic random walks
to produce an improved strategy in the pathological liar game
with a linearly bounded liar.  We also provide
discrepancy bounds of independent interest for a discretized
random walk which we call the ``liar machine''.  Liar games,
introduced by R\'enyi and Ulam \cite{R61,U76}, are played by a
questioner and responder, whom we can Paul and Carole,
respectively, according to tradition; they model search in the
presence of error. The original variant is like ``twenty
questions''  to identify a distinguished element of the search
space, except with lies; while in the pathological variant,
Carole lies as much possible, and Paul tries to preserve at
least one element of the search space. Winning strategies in
liar games correspond to adaptive codes, introduced by
Berlekamp \cite{B64}.  A primary objective in developing
winning strategies for liar games is to optimize the size of a
search space that can be processed given the number of
questions Paul can ask and a constraint on how Carole may lie.
Translated into coding theory language, this objective is to optimize
the size of a message set that can be handled given the number
of bits to be transmitted and a constraint on how noise can
corrupt the transmission.  Berlekamp's codes in \cite{B64} are
adaptive packing codes for error-correction, corresponding to
the original liar game, whereas the pathological liar game,
introduced by the second author and Yan \cite{EY04},
corresponds to adaptive covering codes.  For both the liar game
and adaptive coding viewpoints there is a theoretical size
limit on the search space, called the sphere bound, that
provides the target for optimization, often in terms of a
multiple of the sphere bound.

We combine two ideas to improve the best-known winning strategy
for the pathological liar game with Yes-No questions and a
linearly bounded liar. The first idea is to reduce the
pathological liar game to a chip-moving machine on the integer
line, which we call the liar machine, introduced implicitly by Spencer and
Winkler for the original liar game \cite{SW92}. The second is
to adapt the analysis of deterministic random walks on the
integers, developed by the first author, Doerr, Spencer, and
Tardos \cite{CDST07}, to the time-evolution of the liar
machine, and confirm a winning strategy in the pathological
liar game. Our main results are pointwise and interval
discrepancy bounds on the time-evolution of the liar machine as
compared to random walks on the integers, in Theorems
\ref{thm:ptwiseupperbound} and \ref{thm:intervalupperbound};
and an improved upper bound on the size of the search space for
which Paul can win the pathological liar game with Yes-No
questions and a linearly bounded liar, in Corollary
\ref{thm:pathGameBound}.

\section{Definitions and main results}

\subsection{The liar game and pathological variant}

The R\'enyi-Ulam liar game is an $n$-round 2-person
question-and-answer game on a search space
$[M]:=\{1,\ldots,M\}$.  A fixed integer parameter $e\geq 0$ is
the maximum number of {\em lies} an element of the search space
can accumulate before being disqualified, and the game begins
with an initial function $\ell:\{1,\ldots,M\}\rightarrow
\{0,1,\ldots,e\}$, representing the initial assignment of up to
$e$ lies to each $y\in [M]$.
As elements of $M$ are distinguished only by their number of
lies, we may ignore element labels and consider instead the
initial state vector $x_0=(x_0(0),x_0(1),\ldots,x_0(e))$, where
$x_0(i)=|\{y\in [M]:\ell(y)=i\}|$ is the number of elements of
$[M]$ initialized with $i$ lies. Most often we set
$x_0=(M,0,\ldots,0)$. Paul and Carole play an $n$-round game in
which Paul attempts to discover a distinguished element
$z\in[M]$ of the search space. To start each round, Paul weakly
partitions $[M]$ into two parts by choosing a {\em question}
$(A_0,A_1)$ such that $[M]=A_0\dotcup A_1$, where $\dotcup$
denotes disjoint union. We interpret this choice as the
question, ``Is $z\in A_0$?''.  Carole completes the round by
responding with her {\em answer}, an index $j\in \{0,1\}$. For
each $y\in [M]$, if $y\in A_j$, no additional lie is assigned
to $y$, but if $y\in A_{1-j}$, one additional lie is assigned
to $y$.  Any $y\in [M]$ accumulating $e+1$ lies is {\em
disqualified}. We interpret Carole's answer of $j=0$ as ``Yes''
and of $j=1$ as ``No''. Analogous to the definition of $x_0$,
for each $s=1,\ldots,n$, let the state vector
$x_s=(x_s(0),\ldots, x_s(e))$ record the number of elements
$x_s(i)$ that have $i$ lies after the end of round $s$.  Paul's
question $(A_0,A_1)$ in round $s$ corresponds to a question
vector $a_s=(a_s(0),\ldots,a_s(e))$ with $0\leq a_s(i)\leq
x_{s-1}(i)$ for all $0\leq i\leq e$, by letting $a_s(i)$ count
the number of elements in $A_0$ that have $i$ lies at the end
of round $s-1$. Define the right-shift operator $R$ on any
vector $x=(x(0),\ldots,x(e))$ by $R(x)=(0,x(0),\ldots,x(e-1))$.
Given $x_{s-1}$ and $a_s$,
define \\
\begin{eqnarray*}
Y(x_{s-1},a_s) & := & a_s+R(x_{s-1}-a_s), \\
N(x_{s-1},a_s) & := & x_{s-1}-a_s + R(a_s);
\end{eqnarray*}
and for each $s=1,\ldots,n$, set $x_s=Y(x_{s-1},a_s)$ if Carole
responds $j=0$ (``Yes'') in round $s$, and otherwise
$x_s=N(x_{s-1},a_s)$ if Carole responds $j=1$ (``No'') in round
$s$. Elements $y\in [M]$ that accumulate $e+1$ lies are shifted
out to the right and may be ignored for the rest of the game.
Paul wins the original liar game if
$\sum_{i=0}^{e}x_n(i)\leq 1$, that is, if all but at most one
element are disqualified after $n$ rounds; he wins the
pathological liar game if $\sum_{i=0}^{e}x_n(i)\geq 1$, that
is, if at least one element survives after $n$ rounds. We are
primarily interested in the pathological variant, which may be
interpreted as having a capricious Carole lying to eliminate
elements as quickly as possible, while Paul forms questions to
prevent all elements from being disqualified. We summarize the
pathological liar game as follows.
\begin{definition}
Let $n,M,e\geq 0$ be integers, and let $x=(x(0),x(1),\ldots,x(e))$
be a nonnegative integer vector with $\sum_{i=0}^{e} x(i)=M$.
Define the $(x,n,e)^*_2$-game to be the $n$-round pathological
liar game with Yes-No questions, initial configuration $x$,
and $e$ lies.  We say that Paul can win the $(x,n,e)^*_2$-game
provided there exists a winning strategy for Paul regardless of
Carole's responses.
\end{definition}
In the notation $(x,n,e)^*_2$, use of the asterisk indicates the
pathological variant of the liar game rather than the original.
The subscript 2 means that questions are binary and symmetric
with respect to replacing $a_s$ with $x_{s-1}-a_s$ while preserving
the same two vectors as candidates for $x_s$. This corresponds in
coding theory to the binary symmetric channel assumption; see \cite{EN09}
for a much broader class of
channel assumptions.

\subsection{The liar machine and the linear machine\label{sec:liarLinearDef}}

We define the ``liar machine'' as follows.  Start with some
configuration of chips on the even or odd integers (but not both).  Number the chips
$c_1,c_2,\ldots$ left-to-right.  At each location with, say,
$k$ chips, send $\floor{k/2}$ of the chips one step left, and
$\floor{k/2}$ one step right.  If one chip remains
(because $k$ is odd) we break the tie by sending the highest-indexed $c_j$ one step
left if $j$ is even or one step right if $j$ is odd.

Formally, define the ``starting configuration'' to be a map
$f_0 : \Z \rightarrow \N$ with finite support lying in $2\Z$ or $2\Z+1$.  Then, given
$f_t : \Z \rightarrow \N$, define $\chi_t : \Z \rightarrow
\{-1,0,1\}$ by
\begin{equation} \label{eq:chidef}
\chi_t(j) = \left \{ \begin{array}{ll} 0 & \textrm{ if } f_j \equiv 0 \pmod{2} \\ (-1)^{\sum_{i<j} \chi_t(i)} & \textrm{ if } f_j \equiv 1 \pmod{2}. \end{array} \right .
\end{equation}
Then we define
$$
f_{t+1}(j) = \frac{ f_t(j-1) + f_t(j+1)+ \chi_t(j-1) - \chi_t(j+1) }{2}.
$$

Now, we define the ``linear machine'' by taking $g_0 : \Z
\rightarrow \N$ to be any function.  Let the operator $\cL :
\Z^\Z \rightarrow \Z^\Z$ be defined by
$$
\cL g(j) = \frac{g(j-1)}{2} + \frac{g(j+1)}{2},
$$
and define $g_{t+1} = \cL g_t$.  Then $g_t(j)$ is just the
expected number of chips at location $j$ after a simple random
walk on $\Z$ starting from the configuration $g_0$.  In
particular, we expect $g_t$ and $f_t$ to be relatively close to
one another if $g_0 \equiv f_0$.  Also, define the operator
$\Delta: \Z^\Z \rightarrow \Z^\Z$ by
$$
\Delta f(j) = f(j-1).
$$
It is easy to see that $\cL$ and $\Delta$ are linear, and they
commute with each other.  We write $\delta_j \in \Z^\Z$ for the
function which is $1$ at $j$ and $0$ elsewhere. In order to
consider intervals in a configuration, for a set $S \subset \Z$
and a function $h : \Z \rightarrow \R$, define $h(S) = \sum_{i
\in S} h(i)$.

\subsection{Main results}

Our first two main results are a pointwise and an interval discrepancy bound
in the time-evolution of the liar machine versus the linear machine starting
with the same initial configuration.

\begin{theorem} \label{thm:ptwiseupperbound}
Let $f_0 \equiv g_0$, and define $f_t$ and $g_t$ according to
the evolution of the liar machine and linear machine, respectively, as
described above. Then
$$
|f_t(j) - g_t(j)| < 12  \log t
$$
for all $t \geq 2$, $j \in \Z$.
\end{theorem}

\begin{theorem} \label{thm:intervalupperbound}
Let $I=[a,b] \subset \Z$ and $f_0 \equiv g_0$, and define
$f_t(I)$ and $g_t(I)$ according to the evolution of the liar
machine and linear machine, respectively, as described above.   Then
$$
|f_t(I) - g_t(I)| \leq c^\prime \cdot \left \{
\begin{array}{ll} \sqrt{t} & \textrm{ if } B > \sqrt{t}/2 \\
    B \log (t/B^2) & \textrm{ if } B \leq \sqrt{t}/2, \end{array} \right .
$$
where $B = b-a$ and $c^\prime$ is an absolute constant.
\end{theorem}
In Corollary \ref{cor:mainlowerbound} we prove that Theorems
\ref{thm:ptwiseupperbound} and \ref{thm:intervalupperbound} are
tight up to a constant multiple for a general initial
configuration $f_0$. Corollary \ref{cor:liarGameToLiarMachine}
allows extraction of a winning strategy for the pathological
liar game from the time-evolution of the liar machine, yielding
the following improved bound for the pathological liar game.
\begin{theorem}\label{thm:pathGameBound}Let
$M=\frac{2^n}{\binom{n}{\leq \lfloor fn\rfloor}}
(4/(1-2f))c'\sqrt{\log \log n}(1+o(1))$, where $c'$ is the
constant from Theorem \ref{thm:intervalupperbound}. Then for
$n$ sufficiently large, Paul can win the
$((M,0,\ldots,0),n,\lfloor fn\rfloor)^*_2$-pathological liar
game with $M$ elements and $\lfloor fn\rfloor$ lies on the
binary symmetric channel.
\end{theorem}

We now discuss the improvement provided by Theorem
\ref{thm:pathGameBound}. The previous best known bound on $M$
for $f\in (0,1/2)$ is Theorem 1 of \cite{DP86}, which in our
language bounds the smallest $M$ for which Paul can win the
$((M,0,\ldots,0),n,\lfloor f n \rfloor)_2^*$-game with a
restricted strategy (called ``non-adaptive'' in the literature)
of selecting all questions before any responses from Carole are
available.

\begin{theorem}[Delsarte and Piret]\label{thm:DelsartePiret}
Let $f\in (0,1/2)$. The minimum $M$ for which Paul can win the
$((M,0,\ldots,0),n,\lfloor fn \rfloor)_2^*$-game with the restriction
that all $n$ questions must be formed before any responses from Carole
are available is bounded by
$$
M \leq \left\lceil \frac{2^n}{\binom{n}{\leq \lfloor fn\rfloor}} n
    \log 2\right\rceil.
$$
\end{theorem}
The quantity $2^n/\binom{n}{\leq \lfloor fn\rfloor}$ is called the sphere bound,
and so Theorem \ref{thm:pathGameBound} provides an improved {\em density} in the
best-known minimum $M$, from a linear to sub-logarithmic factor in $n$ times the
sphere bound.  The sphere bound is an immediate lower bound on $M$;
this can be seen by defining an appropriate weight function on the liar game
state which Carole greedily minimizes (cf.~\cite[Lemma 3]{EPY05}).
In Theorem \ref{thm:DelsartePiret}, the ``spheres'' are Hamming balls of radius
$\lfloor fn \rfloor$ that are used to cover the binary discrete hypercube
(Hamming space) of dimension $n$.  The equivalence of
wining strategies in the pathological liar game to coverings of
Hamming space by objects of size $\binom{n}{\leq \lfloor fn\rfloor}$ is
proved in Theorem 3.7 of \cite{EN09}.

We conclude the section by outlining the rest of the paper.
Section \ref{sec:liarMachineProofs} contains the proofs of the
liar machine discrepancy bounds: Theorems
\ref{thm:ptwiseupperbound} and \ref{thm:intervalupperbound},
and Corollary \ref{cor:liarGameToLiarMachine}.  Section
\ref{sec:distribution} proves several technical distributional
facts about the binomial and hypergeometric distributions
needed to bound the distribution of chips in the liar machine
(via discrepancy from the linear machine). Section
\ref{sec:reduction} reduces a strategy for Paul in the
pathological liar game to the liar machine and blends the
preceding results into Theorem \ref{thm:pathGameBound}. Section
\ref{sec:conclusion} contains open questions and closing
remarks.

\section{Proofs of liar machine discrepancy bounds\label{sec:liarMachineProofs}}

The proofs of Theorems \ref{thm:ptwiseupperbound} and
\ref{thm:intervalupperbound} flow directly from the definitions
in Section \ref{sec:liarLinearDef}, and resemble the arguments in \cite{CDST07}. A bound on, and the
bimodality in space of, a term that tracks the discrepancy
between the liar machine and the linear machine is deferred
until Lemma \ref{lemma:painful}.  Next, Lemma
\ref{lem:parityForce} shows that the parity of the number of
chips in the liar machine can be pre-selected for an arbitrary
product of intervals in space and time, by choosing an
appropriate initial configuration. This leads to a
complementary lower bound in Corollary \ref{cor:mainlowerbound}
on discrepancy for a general initial configuration.  We adopt the convention, here and throughout, that $\binom{a}{b}$ is zero unless $a$ and $b$ are nonnegative integers and $b \leq a$.

\begin{proof}[Proof of Theorem \ref{thm:ptwiseupperbound}] Evidently,
$$
f_{t+1} = \cL f_t + \frac{1}{2} (\Delta - \Delta^{-1}) \chi_t.
$$
Therefore, by the linearity of $\cL$ and the fact that it commutes with $\Delta$,
$$
f_t = \cL^{t} f_0 + \frac{1}{2} \sum_{s = 0}^{t-1}
    (\Delta - \Delta^{-1}) \cL^{s} \chi_{t-1-s}.
$$
Since $g_t = \cL^t g_0 = \cL^t f_0$,
\begin{align*}
2|f_t - g_t| &= \left | \sum_{s = 0}^{t-1} (\Delta - \Delta^{-1}) \cL^{s}
    \chi_{t-1-s} \right | \\
& \leq 2 + \sum_{s = 1}^{t-1} \left | (\Delta - \Delta^{-1}) \cL^{s} \chi_{t-1-s} \right |.
\end{align*}
Consider a fixed $s$.  Denote by $z_i$ the $i^\textrm{th}$
element of the support of $\chi_{t-1-s}$, with $z_0$ its
minimal element and $z_{i+1} > z_i$ for each $i$.  Note that the $z_i$ all have the same parity, by our assumption that the chips occupy only even or only odd integers.  Then
\begin{align}
\nonumber (\Delta - \Delta^{-1}) \cL^{s} \chi_{t-1-s} &= (\Delta - \Delta^{-1}) \cL^{s} \sum_{i} (-1)^i \delta_{z_i} \\
\nonumber &= \sum_{i} (-1)^i (\Delta - \Delta^{-1}) \cL^{s} \delta_{z_i} \\
\nonumber &= \sum_{i} (-1)^i (\Delta - \Delta^{-1}) \cL^{s} \Delta^{z_i} \delta_0 \\
\label{eq:sum1} &= \sum_{i} (-1)^i \Delta^{z_i} (\Delta - \Delta^{-1}) \cL^{s} \delta_0.
\end{align}
Note that
$$(\Delta - \Delta^{-1}) \cL^{s}
\delta_0(j) = 2^{-s} \left ( \binom{s}{(s+j-1)/2} - \binom{s}{(s+j+1)/2} \right ).
$$
Therefore, by Lemma \ref{lemma:painful}, $(\Delta - \Delta^{-1}) \cL^{s}
\delta_0$ is bimodal on its support.   This means that the alternating sum
$\sum_{i} (-1)^i \Delta^{z_i} (\Delta - \Delta^{-1}) \cL^{s}
\delta_0$ is bounded by at most four times the maximum (in absolute
value) of the quantity $(\Delta - \Delta^{-1}) \cL^{s}
\delta_0$, since the $z_i$ all have the same parity.  This maximum, by Lemma
\ref{lemma:painful}, is at most $3/s$.
Therefore,
\begin{align*}
|f_t - g_t| &\leq \frac{1}{2} \left ( 12 \sum_{s = 1}^{t-1} \frac{1}{s} + 2 \right ) \\
&\leq 12 \log t.
\end{align*}
\end{proof}

\begin{proof}[Proof of Theorem \ref{thm:intervalupperbound}] Without loss of generality, $I = \{1,\ldots,B\}$.  We may also assume that $B$ is even. Evidently,
$$
f_{t+1}(I) = \left (  \cL f_t(I) + \frac{1}{2} (\Delta - \Delta^{-1}) \chi_t(I) \right ).
$$
Therefore, by the linearity of $\cL$ and the fact that it commutes with $\Delta$,
$$
f_t(I) = \sum_{i = 1}^B \Delta^i \cL^{t} f_0(I) + \frac{1}{2}
    \sum_{s = 0}^{t-1} (\Delta - \Delta^{-1}) \cL^{s} \chi_{t-1-s}(I).
$$
Since $g_t(I) = \cL^t g_0(I) = \cL^t f_0(I)$,
\begin{align*}
2|f_t(I) - g_t(I)| &= \left | \sum_{s = 0}^{t-1} (\Delta - \Delta^{-1}) \cL^{s} \chi_{t-1-s}(I) \right | \\
& \leq 2 + \sum_{s = 1}^{t-1} \left | (\Delta - \Delta^{-1}) \cL^{s} \chi_{t-1-s}(I) \right |.
\end{align*}
Denote by $z_i$ the $i^\textrm{th}$ element of the support of $\chi_{t-1-s}$, with $z_0$ its minimal element and $z_{i+1} > z_i$ for each $i$.  Then
\begin{align}
\nonumber (\Delta - \Delta^{-1}) \cL^{s} \chi_{t-1-s}(I) &= (\Delta - \Delta^{-1}) \cL^{s} \sum_{i} (-1)^i \delta_{z_i}(I) \\
\nonumber &= \sum_{i} (-1)^i (\Delta - \Delta^{-1}) \cL^{s} \delta_{z_i}(I) \\
\nonumber &= \sum_{i} (-1)^i (\Delta - \Delta^{-1}) \cL^{s} \Delta^{z_i} \delta_0(I) \\
\nonumber &= \sum_{i} (-1)^i \Delta^{z_i} (\Delta - \Delta^{-1}) \cL^{s} \sum_{k=1}^B \Delta^{k} \delta_0 \\
\nonumber &= \sum_{i} (-1)^i \Delta^{z_i} \sum_{k=1}^B \Delta^{k} (\Delta - \Delta^{-1}) \cL^{s} \delta_0 \\
\nonumber &= \sum_{i} (-1)^i \Delta^{z_i} (\Delta^{B+1} + \Delta^B - \Delta - 1) \cL^{s} \delta_0 \\
\label{eq:sum2} &= \sum_{i} (-1)^i \Delta^{z_i} (\Delta + 1)(\Delta^{B} - 1) \cL^{s} \delta_0.
\end{align}
Note that $\cL^s \delta_0(j) = 2^{-s} \binom{s}{(s+j)/2}$, so that
$$
(\Delta^B - 1) \cL^{s} \delta_0(j) = 2^{-s} \left ( \binom{s}{(s+j-B)/2} - \binom{s}{(s+j)/2} \right ).
$$
By Lemma \ref{lemma:painful}, when $B = \Omega(\sqrt{s})$, the maximum possible value of the right-hand side is $\Theta(1/\sqrt{s})$; when $B = o(\sqrt{s})$, it is of order
$$
2^{-s} B \cdot \max_{j} \left ( \binom{s}{(s+j-1)/2} - \binom{s}{(s+j+1)/2} \right ) = \Theta(B/s).
$$
Since an alternating sum over a bimodal function like $(\Delta + 1)(\Delta^{B} - 1) \cL^{s} \delta_0$ is bounded by four times the maximum absolute value of that function,
\begin{align*}
|f_t(I) - g_t(I)| &\leq c \sum_{s = 1}^{t-1} \frac{\min(\sqrt{s},B)}{s+1} \\
&\leq c^\prime \cdot \left \{ \begin{array}{ll} \sqrt{t} & \textrm{ if } B > \sqrt{t}/2 \\
                                                B \log (t/B^2) & \textrm{ if } B \leq \sqrt{t}/2, \end{array} \right .
\end{align*}
for some absolute constants $c$ and $c^\prime$.
\end{proof}

\begin{lemma} \label{lemma:painful} There exists constants $c_1$, $c_2$, $c_3$, and $c_4$ so that the following holds for all $s \geq 1$ and even $B > 0$.  Define
$$
h_B(j) = 2^{-s} \left ( \binom{s}{(s+j-B)/2} - \binom{s}{(s+j)/2} \right ).
$$
Then, when $B \geq \sqrt{s}$,
$$
\frac{c_1}{\sqrt{s}} \leq \max_j \left |h_B(j) \right |\leq \frac{c_2}{\sqrt{s}},
$$
where the left-hand inequality holds for all $s \geq S$, some absolute constant.  When $B \leq \sqrt{s}$,
$$
\frac{c_3B}{s} \leq \max_j \left |h_B(j) \right |\leq \frac{c_4B}{s}.
$$
Furthermore, $h_B(j)$ is bimodal on its support.
\end{lemma}
\begin{proof} It is easy to see that
$$
h_2(j) = 2^{-s} \frac{j-1}{s+1} \binom{s+1}{(s+j)/2}.
$$
Therefore,
$$
\frac{h_2(j)}{\Delta^2 h_2(j)} = \frac{(j-1)(s-j+4)}{(j-3)(s+j)},
$$
which equals one when $j^2 - 4j - (s-2) = 0$, i.e., $j = 2 \pm \sqrt{s+2}$.  Then the maximum of $|h_2(j)|$ can be bounded by
\begin{align*}
2^{-s} \frac{|j_{\max}|-1}{s+1} \max_j \binom{s+1}{(s+j)/2} & \leq \frac{1 + \sqrt{s+2}}{s+1} \cdot \frac{1}{\sqrt{3(s+1)/2}} \\
& \leq \frac{1 + 2}{(s+1)\sqrt{2}} < \frac{3}{s}.
\end{align*}
Since $h_B(j) = \sum_{i=0}^{B/2-1} h_2(j-2i)$, it immediately follows that
$$
\max_j \left | h_B(j)\right | \leq \frac{B}{2} \max_j \left | h_2(j) \right | < \frac{3B}{2s},
$$
so we may take $c_4 = 3/2$.  Now, it is clear that
\begin{align*}
\max_j \left |h_B(j)\right | &< \max_j 2^{-s} \cdot \binom{s}{(s+j)/2} \\
&< \frac{1}{\sqrt{s}},
\end{align*}
so we may take $c_2 = 1$.

On the other hand, by a version of the Local Central Limit
Theorem (see, e.g., \cite[Thm.~1.2.1]{L91}), $2^{-s}
\binom{s}{(s+j)/2} = \sqrt{\frac{2}{\pi s}} \cdot e^{-j^2/2s} +
O(s^{-3/2})$, so that we have
$$
h_B(j) = \sqrt{\frac{2}{\pi s}} \cdot e^{-(j-B)^2/2s} - \sqrt{\frac{2}{\pi s}} \cdot e^{-j^2/2s} + O(s^{-3/2}).
$$
Hence, when $B \geq \sqrt{s}$,
\begin{align*}
\max_j \left | h_B(j) \right | &\geq |h_B(0)| \\
&= \left | \sqrt{\frac{2}{\pi s}} \cdot e^{-B^2/2s} -
    \sqrt{\frac{2}{\pi s}} +O(s^{-3/2})\right | \\
&\geq \sqrt{\frac{2}{\pi s}} \left | e^{-1/2} - 1 \right | + O(s^{-3/2})\\
&> \frac{1+o(1)}{4 \sqrt{s}},
\end{align*}
so we may take $c_1 = 1/4$ and $S$ sufficiently large.  Note that the error term $O(s^{-3/2})$ is uniform in $j$ and therefore the $o(1)$ does not depend on $B$.

When $B < \sqrt{s}$,
\begin{align*}
2^s h_B(j) &= \binom{s}{(s+j-B)/2} - \binom{s}{(s+j)/2} \\
&= \binom{s}{(s+j)/2} \left ( \prod_{i=1}^{B/2} \frac{s+j-B+2i}{s-j+2i} - 1 \right ) \\
&= \binom{s}{(s+j)/2} \left ( \prod_{i=1}^{B/2} \left ( 1 + \frac{2j-B}{s-j+2i} \right ) - 1 \right ),
\end{align*}
so we have
\begin{align*}
\max_j h_B(j) &\geq h_B(\sqrt{s}) \\
&= 2^{-s} \binom{s}{(s+\sqrt{s})/2} \left ( \prod_{j=1}^{B/2} \left ( 1 + \frac{2 \sqrt{s}-B}{s-\sqrt{s}+2j} \right ) - 1 \right ) \\
&\geq 2^{-s} \binom{s}{(s+\sqrt{s})/2} \left ( \left ( 1 + \frac{\sqrt{s}}{s} \right )^{B/2} - 1 \right ) \\
&\geq \frac{c_0}{\sqrt{s}} \cdot \frac{B}{2\sqrt{s}} = \frac{c_0 B}{2 s},
\end{align*}
so we can take $c_3 = c_0/2$.

Finally, we have
\begin{align*}
2^s(h_B(j-2) - h_B(j)) &= \binom{s}{(s+j-B)/2-1} - \binom{s}{(s+j-B)/2} \\
& \qquad - \binom{s}{(s+j)/2-1} + \binom{s}{(s+j)/2} \\
&= 2^s (h_2(j-B) - h_2(j))\\
&= \frac{j-B-1}{s+1} \binom{s+1}{(s+j-B)/2} - \frac{j-1}{s+1} \binom{s+1}{(s+j)/2}.
\end{align*}
This quantity is positive when
$$
(j-B-1) \binom{s+1}{(s+j-B)/2} > (j-1) \binom{s+1}{(s+j)/2},
$$
i.e.,
$$
(j-B-1) \prod_{i=1}^{B/2} (s+j-2i+2) > (j-1) \prod_{i=1}^{B/2} (s-j+2i+2).
$$
(We may assume that each term of both products is nonnegative.) When $1 \leq j \leq B+1$, this inequality cannot be satisfied, since the left-hand side is nonpositive and the right-hand side is nonnegative.  When $j > B+1$, the inequality is the same as
$$
\left (1-\frac{B}{j-1} \right ) \prod_{i=1}^{B/2} (s+j-2i+2) > \prod_{i=1}^{B/2} (s-j+2i+2).
$$
The left-hand side is nondecreasing in $j$ and the right-hand side is nonincreasing in $j$, so $h_B(j-2) - h_B(j)$ has at most one change of sign in this regime.  When $j < 1$, we have the condition
$$
\prod_{i=1}^{B/2} (s+j-2i+2) < \left (1 + \frac{B}{j-B-1} \right ) \prod_{i=1}^{B/2} (s-j+2i+2),
$$
where again the left-hand side is nondecreasing in $j$ and the right-hand side is nonincreasing in $j$, so $h_B(j-2) - h_B(j)$ has at most one more change of sign.  Therefore, $h_B(j)$ is bimodal on its support.

\end{proof}

\begin{lemma}\label{lem:parityForce} For each function
$g : \{0,\ldots,N-1\} \times \{0,\ldots,T-1\} \rightarrow \{0,1\}$,
there exists a chip-assignment function $f_0 : \Z \rightarrow \N$
so that, for all $0 \leq n < N$ and $0 \leq t < T$,
$$
f_t(n) \equiv g(n,t) \pmod{2},
$$
where $f_t$ is the state of the liar machine at time $t$ if $f_0$
is its initial state (i.e., at time $t=0$).
\end{lemma}
\begin{proof} We proceed by induction.  For $T = 1$, the result is immediate: we simply set $f_0 \equiv g(\cdot,0)$.  Suppose that the claim holds for $T$, i.e., there exists an $f_0$ so that $f_t$ agrees with $g(\cdot,t)$ in parity for each $t \in \{0,\ldots,T-1\}$.  Now we perform a second induction (on $n$) to show the following claim:
\begin{claim*} For each $n \in \{0,\ldots,N-1\}$, there exists a chip-assignment function $f^{(n)}_0 : \Z \rightarrow \N$ so that, for all pairs $(n^\prime,t)$ with $0 \leq n^\prime < N$ and $0 \leq t < T$ or $0 \leq n^\prime < n$ and $t = T$,
$$
f^{(n)}_t(n^\prime) \equiv g(n^\prime,t) \pmod{2},
$$
where $f^{(n)}_t$ is the state of the liar machine at time $t$ if
$f^{(n)}_0$ is its initial state.
\end{claim*}
Again, the claim is immediate for $n=0$ (given the inductive
hypothesis), since we can just let $f^{(0)}_0 = f_0$ from the top-level induction.
Suppose it holds for $n$.  If $f^{(n)}_T(n) \equiv g(n,T) \pmod{2}$,
then setting $f^{(n+1)}_0 = f^{(n)}_0$ clearly suffices to prove
the claim for $n+1$.  If, however, $f^{(n)}_T(n) \not\equiv g(n,T)
\pmod{2}$, then define $f^{(n+1)}_0$ by
$$
f^{(n+1)}_0 (k) = \left \{ \begin{array}{ll} f^{(n)}_0(k)
    & \textrm{if } k \neq n+T \\ f^{(n)}_0(k) + 2^T
    & \textrm{if }k=n+T. \end{array} \right .
$$
Then $f^{(n+1)}_t(k) \equiv f^{(n)}_t(k) \pmod{2}$ for $t < T$ and $0 \leq k < N$, since
the ``new'' $2^T$ chips placed at site $n+T$ at time $t = 0$ are split
exactly in half at each time $t < T$, so that
$2 | 2^{T-t} | f^{(n+1)}_t(k) - f^{(n)}_t(k)$ for all $t < T$.  For
$t = T$ and $k < n$, $f^{(n+1)}_t(k) = f^{(n)}_t(k)$, since the ``new''
chips can only occupy sites in $[n+T-t,n+T+t]$ at time $t$, which
for $T=t$ is the interval $[n,n+2T]$ not containing $k$.  Finally,
there is one chip added to site $n$ at time $T$, i.e.,
$f^{(n+1)}_T(n) = f^{(n)}_T(n) + 1$, because exactly one of the $2^T$
``new'' chips makes it to site $n$ after $T$ steps.  This means,
in particular, that $f^{(n+1)}_T(n) \equiv g(n,T) \pmod{2}$, completing
the induction.
\end{proof}

This ``parity forcing'' lemma implies that it is possible
to set the function $\chi_t(j)$ for any finite space-time interval to whatever we wish.
We may then conclude that Theorems \ref{thm:ptwiseupperbound} and
\ref{thm:intervalupperbound} are tight.

\begin{cor} \label{cor:mainlowerbound} Fix $T$, a nonnegative integer, and $N$, and integer.
There exists an $f_0 : \Z \rightarrow \N$ so that, letting
$g_0 \equiv f_0$, and defining $f_t$ and $g_t$ according to the
evolution of the liar machine and linear machine, respectively, we have
$$
|f_T(N) - g_T(N)| = \Omega(\log T).
$$
Fix an interval $I$ of any given length $B$.  Then there exists
an $f^\prime_0 : \Z \rightarrow \N$ so that, letting $g^\prime_0 \equiv f^\prime_0$,
and defining $f^\prime_t$ and $g^\prime_t$ according to the evolution of the
liar machine and linear machine, respectively, we have
$$
|f^\prime_T(I) - g^\prime_T(I)| = \Omega \left( \left \{
\begin{array}{ll} \sqrt{T} & \textrm{ if } B > \sqrt{T}/2 \\
    B \log (T/B^2) & \textrm{ if } B \leq \sqrt{T}/2 \end{array} \right .\right).
$$
\end{cor}
\begin{proof}
The same argument applies for both claims: we can set the
number of chips at each location and time so that the sums in
the proof of Theorems \ref{thm:ptwiseupperbound} and
\ref{thm:intervalupperbound} are maximized, in view of the lower bounds given by Lemma 6.  In the first case, let
$\chi : \{N-T,\ldots,N+T\} \times \{0,\ldots,T-1\} \rightarrow
\{-1,0,1\}$ be chosen to maximize the sum (\ref{eq:sum1}); in the second case, let $\chi : \{\min(I)-T,\ldots,\max(I)+T\} \times \{0,\ldots,T-1\} \rightarrow \{-1,0,1\}$ be chosen to maximize the sum (\ref{eq:sum2}).  Note that this
requires that $\chi$ alternate in sign on the support of its first
argument.  Let $m_t$ be the minimum element of the support of
$\chi(\cdot,t)$. Define
$$
g(k,t) = \left \{ \begin{array}{ll}
\frac{1 - \chi(m_t,t)}{2} & \textrm{if }k = N-T-1 \\
|\chi(k,t)|               & \textrm{if }N-T \leq k \leq N+T \\
0                      & \textrm{otherwise},
\end{array} \right .
$$
in the first case, or else
$$
g(k,t) = \left \{ \begin{array}{ll}
\frac{1 - \chi(m_t,t)}{2} & \textrm{if }k = \min(I)-T-1 \\
|\chi(k,t)|               & \textrm{if }\min(I)-T \leq k \leq \max(I)+T \\
0                      & \textrm{otherwise},
\end{array} \right .
$$
in the second case.  Then, we may obtain the desired $f_0$ by
applying the preceding lemma to $g$.  Since the (possible) chip
at $k = N-T-1$ or $k = \min(I)-T-1$ can never even reach the
site $N$ or any of $I$ before time $T$, the relevant sums are
unaffected by this small modification.  However, the presence
of such a chip when appropriate ensures that $\chi_t(j) =
\chi(j,t)$ for each $(j,t) \in \{N-T,\ldots,N+T\} \times
\{0,\ldots,T-1\}$ or  $(j,t) \in \{\min(I)-T,\ldots,\max(I)+T\}
\times \{0,\ldots,T-1\}$ (where $\chi_t(j)$ is as defined in
(\ref{eq:chidef})).
\end{proof}

\section{Liar machine distributional bound\label{sec:distribution}}

We need several technical facts to obtain to obtain lower
bounds for the configuration of chips in the time-evolution of
the liar machine. Lemma \ref{cutoffLemma} shows that the
cumulative distribution of the binomial random variable drops
off sharply just below where it is evaluated.
Lemma \ref{lem:relativeCDF} shows that the ratio of the
evaluations at the same relative position of the cumulative
distributions of binomial random variables with a similar
number of trials is not too small.  This is needed to bound the
left tail of the liar machine from below.
Because for Theorem \ref{ref:liarMachineBound} we will run $n$
steps of the liar machine in two stages of $n_1$ and $n_2$
steps, respectively, terms of a hypergeometric distribution
arise. Theorem \ref{thm:hyperGeom} quotes a result on the
closeness of the median to the mean of a generalized
hypergeometric distribution from \cite{Si01}, specialized to
the hypergeometric distribution in Corollary \ref{mmCor}. Then
in Proposition \ref{halfprop} we show that for $r$ sufficiently
close to but below the mean $\mu$, asymptotically almost half
of the hypergeometric distribution lies below $r$.  This allows
transferring from a partial sum of hypergeometric distributions
in $n_1$ and $n_2$ to that of the binomial distribution in $n$,
in Proposition \ref{prop:HyperGeomApprox}.  This last result is
critical for Theorem \ref{ref:liarMachineBound} in negotiating
a lower bound on the number of chips between two stages in the
time-evolution of the liar machine, so that at least one chip
survives in a prescribed interval after $n$ rounds.

Throughout the section, we use the following notation. Let $n
\rightarrow \infty$, fix $f \in (0,1/2)$, and set
$$
n_1 = n - \floor{\frac{4}{(1 - 2f)^2} \log \log n}
$$
and $n_2 = n - n_1$. The numbers of rounds in the first and
second stages of the $n$-round liar machine, are $n_1$ and
$n_2$, respectively.
Define $F=\lfloor f n\rfloor$, $F_1 = \lfloor f n_1\rfloor$,
and $F_2 = F-F_1$.

\begin{lemma} \label{cutoffLemma}
For any integer sequence $n_3 = n_3(n) \rightarrow \infty$,
there is a function $\epsilon(n,f)$ with $\lim_{n \rightarrow
\infty} \epsilon(n,f) = 0$ so that
$$
\sum_{i = F - n_3}^{F} \frac{\binom{n}{i}}{\binom{n}{\leq F}}
    \geq 1 - \epsilon(n,f).
$$
\end{lemma}
\begin{proof}
Note that
\begin{align*}
\frac{\binom{n}{F - t}}{\binom{n}{F}} &= \frac{F!(n-F)!}{(F-t)!(n-F+t)!} \\
&\leq \frac{F^t}{(n-F+1)^t}\\
&\leq \frac{(fn)^t}{((1-f)n)^t} = \left ( \frac{f}{1-f} \right)^t.
\end{align*}
Therefore,
\begin{align*}
\sum_{i = 0}^{F-n_3} \binom{n}{i} & \leq \sum_{i = 0}^{F-n_3}
    \binom{n}{F} \left ( \frac{f}{1-f} \right)^{F-i} \\
&\leq \binom{n}{F} \sum_{j = n_3}^{\infty} \left ( \frac{f}{1-f}
    \right)^j \\
&\leq \binom{n}{\leq F} \left ( \frac{f}{1-f} \right)^{n_3}
    \cdot \frac{1-f}{1-2f}.
\end{align*}
It follows that
\begin{align*}
\sum_{i = F-n_3}^F \frac{\binom{n}{i}}
    {\binom{n}{\leq F}} &= 1 - \sum_{i = 0}^{F-n_3-1}
    \frac{\binom{n}{i}}{\binom{n}{\leq F}}\\
& \geq 1 - \left ( \frac{f}{1-f} \right)^{n_3} \cdot \frac{1-f}{1-2f},
\end{align*}
which clearly tends to $1$ as $n \rightarrow \infty$, since $f < 1/2$ implies $f/(1-f) < 1$.
\end{proof}

\begin{lemma}\label{lem:relativeCDF}
There exists a function $\delta(n,f)$ with
$\lim_{n \rightarrow \infty} \delta(n,f) = 0$ so that
$$
\frac{2^n}{\binom{n}{\leq \floor{fn}}}
    \cdot \frac{\binom{n_1}{\floor{fn_1}}}{2^{n_1}}
    \geq (\log n)^{2 - \delta(n,f)}.
$$
\end{lemma}
\begin{proof} First of all, note that
\begin{align*}
\frac{2^n}{\binom{n}{\leq F}} \cdot \frac{\binom{n_1}{F_1}}{2^{n_1}} =
    2^{n-n_1} \frac{\binom{n_1}{F_1}}{\binom{n}{F}}
    \cdot \frac{\binom{n}{F}}{\binom{n}{\leq F}}.
\end{align*}
Denote by $A$, $B$, and $C$ the three factors on the right-hand
side. Since
$$
n - n_1 = \floor{\frac{4}{(1-2f)^2} \log \log n} \geq
    \frac{4}{(1-2f)^2} \log \log n -1,
$$
we have
$$
A \geq \frac{1}{2} (\log n)^{4 \log 2/(1-2f)^2}.
$$
Then, applying the estimates from the proof of Lemma
\ref{cutoffLemma},
\begin{align*}
\binom{n}{\leq F} &= \binom{n}{F} \sum_{t = 0}^F
    \frac{\binom{n}{F-t}}{\binom{n}{F}} \\
&\leq \binom{n}{F} \sum_{t = 0}^\infty \left (\frac{f}{1-f}
    \right )^t = \binom{n}{F} \cdot \frac{1-f}{1-2f},
\end{align*}
so that $C \geq \frac{1-2f}{1-f}$.  Now, we use the fact that
$\binom{n}{\alpha n} = 2^{H(\alpha)n + O(1)}/\sqrt{n}$, where
$H(x) = -x \log_2 x - (1-x) \log_2 (1-x)$ is the entropy
function.  We may therefore write $B$ as
\begin{align*}
\frac{\binom{n_1}{F_1}}{\binom{n}{F}}
    &= 2^{H(f)(n_1 - n) + O(1)} \cdot \frac{\sqrt{n}}{\sqrt{n_1}}\\
&\geq \beta \left ( 2^{\frac{-4 H(f)}{(1-2f)^2} \log \log n} \right )\\
&= \beta (\log n)^{-\frac{4 H(f)\log 2}{(1-2f)^2}},
\end{align*}
where $\beta>0$ is an absolute constant. Combining these
bounds, we have
\begin{align*}
ABC \geq \frac{\beta}{2} \frac{1-2f}{1-f} (\log n)^{\frac{4 \log 2}
    {(1-2f)^2}(1-H(f))}.
\end{align*}
It is easy to check that $\frac{4 \log 2}{(1-2f)^2}(1-H(f)) > 2$ for all $f \in (0,1/2)$, from which the desired bound follows.
\end{proof}

\begin{lemma}
$$ \log \binom{n}{\leq \floor{fn}} = \Theta(n),$$
where the implicit constant depends on $f$.
\end{lemma}
\begin{proof} This follows immediately from the estimate
$$
\binom{n}{\floor{fn}} = 2^{H(f) n + O(\log n)}
$$
as in the proof of Lemma \ref{lem:relativeCDF}.
\end{proof}

The following result appears in \cite{Si01}.

\begin{theorem}\label{thm:hyperGeom}
Let an urn contain $R$ red balls and $B$ black balls.
Suppose each red ball has weight $w_\circ$ and each black has weight
$w_\bullet$. Suppose that the balls are selected one-by-one without
replacement where each as yet unselected ball is given a probability
of being selected at the next round that equals its current fraction
of the total weight of all unselected balls. Suppose $r$ and $b$
satisfy $r = R(1 - e^{-w_\circ \rho})$ and $b = B(1 - e^{-w_\bullet \rho})$,
for some fixed $\rho > 0$.  Let $r + b$ balls be drawn from the urn
as prescribed. Let $X_\circ$ be the number of red balls selected
by this random process, and let $X_\bullet$ be the number of black,
so that $X_\circ + X_\bullet = r + b$. Then $r^\prime = \ceil{r}$
or $\floor{r}$ and $b^\prime = \ceil{b}$ or $\floor{b}$ are the
medians of $X_\circ$ and $X_\bullet$, respectively.
\end{theorem}

By taking $w_\circ = w_\bullet$, i.e., $r/b = R/B$, this result
gives the median of the hypergeometric distribution.  If we let
$r+b = T$ be the total number of balls drawn, then this gives
$b = BT/(R+B)$, i.e., the mean of $X_\circ$.  Hence, we have
the following Corollary.

\begin{cor} \label{mmCor} If $\mu$ is the mean and $m$ the median
of a hypergeometric distribution, then $m = \ceil{\mu}$ or $m = \floor{\mu}$.
\end{cor}

\begin{prop} \label{halfprop} Let $0 \leq r \leq fn_2$.  Suppose
that $fn_1 + r$ elements are drawn uniformly at random (without
replacement) from a set $S = S_1 \dotcup S_2$ with
$|S_1| = n_1$ and $|S_2| = n_2$.  Let $X$ denote the number of
such elements in $S_2$.  If $n_1,n_2 \rightarrow \infty$, there
is some function $h: \N \rightarrow \N$ with $h = \omega(1)$ so
that, for $r \geq fn_2 - h(n)$, we
have
$$
\Pr \left ( X \leq r \right ) \geq 1/2 - o(1).
$$
\end{prop}
\begin{proof} $X$ follows a hypergeometric distribution with parameters $n = n_1 + n_2$, $n_2$, and $R = fn_1 + r$.  Its expectation is therefore given by $\mu = \frac{n_2}{n}(fn_1+r) = n_2 R/n$.  Writing $p(k)$ for the probability that $X = k$, note that
\begin{align*}
p(k) &= \frac{\binom{n_2}{k}\binom{n_1}{R-k}}{\binom{n}{R}}.
\end{align*}
When $k = \mu + \Delta$, we have
\begin{align*}
\frac{p(k)}{p(k-1)} &= \frac{\binom{n_2}{k}\binom{n-n_2}{R-k}}{\binom{n_2}{k-1}\binom{n-n_2}{R-k+1}} \\
&= \frac{(k-1)!(n_2-k+1)!(R-k+1)!(n-n_2-R+k-1)!}{k!(n_2-k)!(R-k)!(n-n_2-R+k)!} \\
&= \frac{(n_2-k+1)(R-k+1)}{k(n-n_2-R+k)} \\
&= \frac{(n_2 - \frac{Rn_2}{n} - \Delta + 1)(R - \frac{Rn_2}{n} - \Delta + 1)}{(\frac{Rn_2}{n} + \Delta)(n - n_2 - R + \frac{Rn_2}{n} + \Delta)}\\
&= \frac{(1 - \frac{R}{n} + \frac{1-\Delta}{n_2})(1 - \frac{n_2}{n}+\frac{1-\Delta}{R})}{(1 - \frac{n_2}{n} - \frac{R}{n} + \frac{Rn_2}{n^2} + \frac{\Delta}{n})(1 + \frac{\Delta n}{Rn_2})}.
\end{align*}
Then,
\begin{align*}
\frac{p(k)}{p(k-1)} - 1 &= \frac{(1 - \frac{R}{n} + \frac{1-\Delta}{n_2})(1 - \frac{n_2}{n}+\frac{1-\Delta}{R})}{(1 - \frac{n_2}{n} - \frac{R}{n} + \frac{Rn_2}{n^2} + \frac{\Delta}{n})(1 + \frac{\Delta n}{Rn_2})} - 1\\
&= \frac{O ( \Delta n^2)}{n_2Rn (1 - \frac{n_2}{n} - \frac{R}{n} + \frac{Rn_2}{n^2} + \frac{\Delta}{n})(1 + \frac{\Delta n}{Rn_2})}.
\end{align*}
Since $n_1+n_2=n$, it follows that $n_2 + R \leq n$.  Therefore,
\begin{align*}
\frac{p(k)}{p(k-1)} - 1 &=O( \Delta ) \frac{n^2}{n_2Rn (\frac{Rn_2}{n^2}+ \frac{\Delta}{n})(1 + \frac{\Delta n}{Rn_2})} \\
&= O( \Delta ) \cdot \frac{n^2}{(n_2R + \Delta n)^2}\\
&= O( \Delta ) \cdot \left ( \frac{1}{\Delta + n_2R/n} \right )^2.
\end{align*}
The quantity $z/(z+a)^2$ is maximized when $z = a$, i.e, $z/(z+a)^2 = (4a)^{-1}$, so
$$
\frac{p(k)}{p(k-1)} - 1 = O \left ( \frac{n}{n_2R} \right ) = O \left ( \frac{n}{n_2n_1} \right ) = o(1).
$$
Therefore, as $n \rightarrow \infty$, the number of $k$'s so that $p(k)$ is within $1+o(1)$ of $p(\mu)$ grows without bound.  This implies that $p(k) = O(1/g(n))$ for some function $g : \N \rightarrow \N$ with $g = \omega(1)$ and all $k$.  If we let $h(n) = \sqrt{g(n)}$, the total probability that $r \leq X \leq \mu$ is $O(1/\sqrt{g(n)}) = o(1)$.  Since, by Corollary \ref{mmCor}, $\ceil{\mu}$ or $\floor{\mu}$ is the median of the hypergeometric distribution, this implies that $\Pr(X \leq r) \geq 1/2 + o(1)$.
\end{proof}

\begin{prop}\label{prop:HyperGeomApprox}
For $n$ tending to infinity and a fixed $f \in (0,1/2)$,
$$
\sum_{k=F_1}^{F} \sum_{s=F_1}^{k} \binom{n_1}{s}
    \binom{n_2}{k-s} = \left ( \frac{1}{2} + o(1) \right )
    \sum_{k=0}^{F} \binom{n}{k}
$$
\end{prop}
\begin{proof}
Let $n_3 = \lceil \sqrt{2 F_2}\rceil$. By Proposition
\ref{halfprop}, we have
$$
\frac{\sum_{s=F_1}^{k} \binom{n_1}{s}\binom{n_2}{k-s}}
    { \binom{n}{k}} \geq \frac{1}{2} - o(1),
$$
for $k \in [F-n_3,F]$ since the left-hand quantity represents
the probability, if a set of $k = F_1 + r$ elements is drawn
uniformly at random, $F_2 - n_3 \leq r \leq F_2$, that at most
$r$ of the points will be taken from the last $n_2$ of all $n =
n_1 + n_2$ elements. Therefore, by the above and then by Lemma
\ref{cutoffLemma},
\begin{align*}
\sum_{k=F_1}^{F} \sum_{s=F_1}^{k} \binom{n_1}{s}
    \binom{n_2}{k-s} &= \sum_{k=F-n_3}^{F} \sum_{s=F_1}^{k}
    \binom{n_1}{s}\binom{n_2}{k-s} \\
& \qquad + \sum_{k=F_1}^{F-n_3-1} \sum_{s=F_1}^{k}
    \binom{n_1}{s}\binom{n_2}{k-s}\\
& \geq \left ( \frac{1}{2} - o(1) \right ) \sum_{k=F-n_3}^{fn}
    \binom{n}{k} \\
& \geq \left ( \frac{1}{2} - o(1) \right ) \left ( 1 - o(1) \right )
    \sum_{k=0}^{F} \binom{n}{k} \\
& \geq \left ( \frac{1}{2} - o(1) \right ) \sum_{k=0}^{F} \binom{n}{k}.
\end{align*}
\end{proof}

\section{Reduction from liar machine to the pathological
liar game\label{sec:reduction}}


We now consider the alternating-question strategy for Paul, and
show that Carole has no better response strategy than always
assigning a lie to each of the odd-numbered chips. The
time-evolution of the chips under these question-and-response
strategies is equivalent, by
Cor.~\ref{cor:liarGameToLiarMachine}, to the liar machine.  We
then combine results of the previous sections to prove Theorem
\ref{thm:pathGameBound} on parameters for which Paul can win.

\begin{definition}[Position vector]
Given the state vector $x=(x(0),\ldots,x(e))$ of a liar game
with $M$ elements,
the {\em position vector} $u=u(x)=(u(1),u(2),\ldots,u(M))$
corresponding to $x$ is defined by
$$ u(j):=\min\left\{k:\sum_{i=0}^k x(i)\geq j\right\}. $$
\end{definition}
\begin{exm}
The position vector of a state vector essentially labels the
$M$ elements tracked by the state vector from left to right,
and records as $u(j)$ the number of lies associated with the
$j$th element.
\begin{eqnarray*}
\mathrm{If}\quad x & = & (2,0,1,3,0), \quad \mathrm{then}\\
    u=u(x) & = & (0,0,2,3,3,3).
\end{eqnarray*}
\end{exm}

Position vectors are monotonic increasing, and provided the
maximum number of lies is available (from context, for
example), the state vector can be recovered from the position
vector.
We analyze the round-by-round evolution of state vectors by
comparing their corresponding position vectors under the weak
majorization partial order, presented for analysis of the
original liar game by \cite{SW92}.

\begin{definition}[Partial order on position vectors] Let
$M\in\mathbb{Z}^+$, and let
$$U=\{(u(1),\ldots,u(M))\in
\mathbb{N}^M:u(1)\leq \cdots \leq u(M)\}$$
be the set of position vectors with $M$ entries. For $u,v\in
U$, we define the partial order $u\leq v$ provided for all
$1\leq k\leq M$, $\sum_{j=1}^k u(j)\leq \sum_{j=1}^k v(j)$.
\end{definition}

\begin{exm}
The partial order on position vectors gives $(0,2,2)\leq
(1,1,2)\leq (1,2,2)\leq (2,2,2)$.
\end{exm}

In order to analyze position vectors within the partial order,
it will be convenient to continue tracking disqualified
elements, with position at least $e+1$, in the position vector.
We do this with the understanding that disqualified elements
are dropped when converting back to the state vector. The {\em
alternating question} for Paul puts all elements tracked by an
even (odd) index in the position vector $u$ into $A_0$ ($A_1$).
The number of lies associated with each element is easily read
from the position vector. Carole's response either assigns an
additional lie to the elements indexed by the odd positions, to
obtain the new position vector $\odd(u)$, or assigns an
additional lie to the elements indexed by the even positions,
to obtain the new position vector $\even(u)$.

\begin{definition}[$\odd(u)$ and $\even(u)$]\label{def:oeu}
Given the position vector $u=(u(1),\ldots,u(M))$, define the
position vector $\odd(u)$ to be the result of sorting
$(u(1)+1,u(2),u(3)+1,u(4),\ldots,u(M)+(M\mod 2))$ in
nondecreasing order, and define the position vector $\even(u)$
to be the result of sorting
$(u(1),u(2)+1,u(3),u(4)+1,\ldots,u(M)+(M+1\mod 2))$ in
nondecreasing order.
\end{definition}

The following two properties appear in the proof of Lemma 2 of
\cite{SW92}.  There is a minor error in the proof of the second
property which we describe and correct after stating the lemma.

\begin{lemma}\label{lem:eUeV}
Let $u$ and $v$ be position vectors of liar games with the same
number of elements on the
binary symmetric channel. Then\\
\noindent (1) \ $\even(u) \leq \odd(u)$, and\\
\noindent (2) \ If $u\leq v$, then $\even(u)\leq \even(v)$.
\end{lemma}

The proof of (1) is a straightforward verification. We defer
the proof of (2) until after describing how to transform $u$
into $v$ in manageable steps.
Close inspection will reveal that the proof in \cite{SW92} does
not find a transformation from $u=(0,1,2)$ to $v=(1,1,1)$; a
successful procedure is as follows.

\begin{alg}
\label{alg:uTov} (Transformation of $u\rightarrow u' \leq v$
with $u<u'$.)

\noindent Input: Position vectors $u=(u(1),\ldots,u(M))$ and
$v=(v(1),\ldots,v(M))$ with $u < v$.\\
Output: A position vector $u'$ with $u<u'\leq v$.\\
\noindent 0. \ Initialize $u'=u$.\\
\noindent 1. \ If $\sum_{i=1}^M u(i) < \sum_{i=1}^M v(i)$, then
set $u'(M)=u(M)+
\sum_{i=1}^M v(i) - \sum_{i=1}^M u(i)$.\\
\noindent 2. \ Otherwise, if $\sum_{i=1}^M u(i) = \sum_{i=1}^M v(i)$:\\
\phantom{mmm}\noindent 2a. \ Maximize $j$ such that $u(j)<v(j)$.\\
\phantom{mmm}\noindent 2b. \ Minimize $k>j$ such that $u(k)>v(k)$.\\
\phantom{mmm}\noindent 2c. \ Set $u'(j)=u(j)+1$ and $u'(k)=u(k)-1$.\\
(By design of $j$ and $k$, $u(j)<v(j),
u(j+1)=v(j+1),\ldots,u(k-1)=v(k-1),u(k)>v(k)$. Furthermore,
$u'$ is already in nondecreasing order.)
\end{alg}
\begin{proof}
The algorithm is easy to verify for $u'$ produced by Step 1.
Suppose Step 2 is executed.  Step 2a certainly produces a
maximum $j$: $u<v$ implies that
$\sum_{i=1}^{\ell}u(i)<\sum_{i=1}^{\ell}v(i)$ for some $\ell$,
and so at least one choice for $j$ with $u(j)<v(j)$ exists.
Step 2b produces a minimum $k$: using the $j$ from Step 2a and
combining the inequalities $\sum_{i=1}^{j-1}u(i)\leq
\sum_{i=1}^{j-1}v(i)$, $u(j)<v(j)$, and
$\sum_{i=1}^Mu(i)=\sum_{i=1}^M v(i)$ yields
$\sum_{i=j+1}^Mu(i)>\sum_{i=j+1}^Mv(i)$; and so there is at
least one choice of $k$ for which $u(k)>v(k)$. For all indices
$i$ strictly between $j$ and $k$, $u(i)<v(i)$ is impossible by
maximality of $j$, and $u(i)>v(i)$ is impossible by minimality
of $k$. The middle entries of $u$ and $v$ are as follows:
\begin{equation}\label{eqn:relationString}
u(j)<v(j),
u(j+1)=v(j+1),\ldots,u(k-1)=v(k-1),u(k)>v(k).
\end{equation}
It remains to verify that $u<u'\leq v$ for $u'$ constructed in
Step 2c. Already $u'$ is in nondecreasing order, by definition
of $u'$, inspection of \eqref{eqn:relationString}, and noting
that $u(j)<u'(j)\leq v(j)$ and $v(k)\leq u'(k)<u(k)$.
Furthermore, for $1\leq \ell\leq j-1$,
$\sum_{i=1}^{\ell}u(i)=\sum_{i=1}^{\ell}u'(i)\leq
\sum_{i=1}^{\ell}v(i)$. With $u(j)+1=u'(j)\leq v(j)$, we have
$1+\sum_{i=1}^{j}u(i)=\sum_{i=1}^{j}u'(i)\leq
\sum_{i=1}^{j}v(i)$. Since $u(i)=u'(i)=v(i)$ for all $j+1\leq
i\leq k-1$, we have $1+\sum_{i=1}^{\ell}u(i)=
\sum_{i=1}^{\ell}u'(i)\leq \sum_{i=1}^{\ell}v(i)$ for all
$j+1\leq \ell\leq k-1$.  With $v(k)\leq u'(k)=u(k)-1$, we have
$\sum_{i=1}^{k}u(i)= \sum_{i=1}^{k}u'(i)\leq
\sum_{i=1}^{k}v(i)$. Since $u\leq v$ and $u'(i)=u(i)$ for
$i>k$, $\sum_{i=1}^{\ell}u(i)=\sum_{i=1}^{\ell} u'(i)\leq
\sum_{i=1}^{\ell}v(i)$ for $k+1\leq \ell\leq M$.
\end{proof}


\begin{proof}[Proof of Lemma \ref{lem:eUeV} Part (2)]
Iterative application of Algorithm \ref{alg:uTov} produces a
sequence of position vectors $u=u_0<u_1<\cdots<u_t=v$. The
sequence terminates because there are a bounded number of
position vectors satisfying the precondition $\sum_{i=1}^M u(i)
= \sum_{i=1}^M v(i)$ to execute Step 2 of the algorithm. Now
let $0\leq s<t$ and consider $u_s<u_{s+1}$. If $u_{s+1}$ was
created by applying Step 1 of the algorithm to $u_s$ (thereby
forcing $s=0$), then $\even(u_s)\leq \even(u_{s+1})$ is easy to
verify.

Otherwise Step 2 created $u_{s+1}$ from $u_s$. Inspection of
\eqref{eqn:relationString} reveals that
$u_s(j)<u_{s+1}(j)=u_s(j)+1\leq u_{s+1}(k)=u_s(k)-1<u_s(k)$.
Ignoring for a moment the $j$th and $k$th entries of $u_s$ and
$u_{s+1}$, and applying $\even$ to all other entries and then
resorting, we have the following identical structure for
$\even(u_s)$ and $\even(u_{s+1})$:
$$
\begin{array}{cc|ccc|cc}
 \cdots & \leq u_s(j)+\chi_{2|j}
    & \geq u_s(j)+1 +\chi_{2|j}
    & \cdots & \leq u_s(k)-1 + \chi_{2|k}
    & \geq u_s(k) + \chi_{2|k} & \cdots\, .
\end{array}
$$
Here, $\chi_{2|j}$ ($\chi_{2|k}$) equals 1 if 2 divides $j$
($k$) and equals 0 otherwise; the vertical separators denote
that smaller entries lie to the left and larger to the right.
Now we can see that $\even(u_s)$ is the same as inserting
$u_s(j)+\chi_{2|j}$ and $u_s(k)+\chi_{2|k}$ from left to right
at the two separators without need for resorting.  Similarly,
$\even(u_{s+1})$ is the same as inserting $u_s(j)+1+\chi_{2|j}$
and $u_s(k)-1+\chi_{2|k}$ from left to right at the two
separators without need for resorting. With this observation it
is simple to verify that $\even(u_s) < \even(u_{s+1})$.

Since $s$ was arbitrary in the preceding argument, we have
$\even(u)=\even(u_0)<\even(u_1)<\cdots<\even(u_t)=\even(v)$,
and so combined with the case $t=0$ for which
$\even(u)=\even(v)$, Part (2) of the lemma holds.
\end{proof}

\begin{cor}\label{cor:oUoV}
Let $u$ and $v$ be position vectors of liar games with the same
number of elements on the binary symmetric channel. If $u\leq
v$, then $\odd(u)\leq \odd(v)$.
\end{cor}
\begin{proof}
We use a trick to piggyback on Lemma \ref{lem:eUeV} Part (2).
Set $u'=(-2,u(1),\ldots,u(M))$ and $v'=(-2,v(1),\ldots,v(M)$
and observe that $u\leq v$ implies $u'\leq v'$.  The first
entry of $u'$ and of $v'$ is sufficiently separated, and so
$\even(u')=(-2,\odd(u))$ and $\even(v')=(-2,\odd(v))$. Applying
Lemma \ref{lem:eUeV}  to $u'$ and $v'$ yields $\even(u')\leq
\odd(v')$.  As $\even(u')(1)=\even(v')(1)=-2$, this forces
$\odd(u)\leq \odd(v)$.
\end{proof}

Next we show that when Paul's strategy is to always ask the
alternating question, Carole's best possible response strategy
in the pathological liar game is to move the odd-numbered
elements.  This will provide an upper bound on the minimum
number of elements required for Paul to have a winning strategy
in the $(x,n,e)^*_2$-game.

\begin{theorem}\label{thm:CaroleStrat}
Let $x$ be an initial state vector, and $n,e\in \mathbb{N}$.
Assume that Paul always asks the alternating question. In the
$(x,n,e)^*_2$-game, Carole's best strategy is to move the
odd-numbered elements.
\end{theorem}
\begin{proof}
Let $u_s$ be the position vector after $s$ rounds of the game,
where $u_0$ is the position vector corresponding to the initial
state vector $x$. Carole wins the $(x,n,e)^*_2$-game iff
$u_n(1)>e$. Consider the $2^n$ leaves of the strategy tree of
the game determined by every possible length $n$ sequence of
choices for Carole to select $\odd(u_s)$ or $\even(u_s)$ to
complete round $s+1$. Thus $\odd^n(u_0)$ is the leaf
corresponding to Carole always moving the odd elements. It
suffices to show that $\odd^n(u_0)\geq v$ for all other leaves
$v$ of the strategy tree. We prove this by induction on $n$.
The base case $n=1$ is provided by Lemma \ref{lem:eUeV} Part
(1).  Now let $0<s<n$, assume that $v$ is a position vector
after $s$ rounds, and assume that $v\leq \odd^s(u_0)$. By
Corollary \ref{cor:oUoV}, $\odd(v)\leq
\odd(\odd^s(u_0))=\odd^{s+1}(u_0)$, and by Lemma \ref{lem:eUeV}
Part (1) and transitivity, $\even(v)\leq \odd^{s+1}(v)$. All
position vectors after $s+1$ rounds are obtained by applying
$\odd$ or $\even$ to a position vector after $s$ rounds, and so
the induction succeeds.
\end{proof}

%
%
%

By a simple transformation, Carole's odd response strategy is
equivalent to the time-evolution of the liar machine.

\begin{cor}\label{cor:liarGameToLiarMachine}
Let the liar machine have initial configuration $f_0$ with $M$
chips at the origin and none elsewhere.  If
$\sum_{i=-n}^{-n+2e}f_n(i)\geq 1$, then Paul can win the
$((M,0,\ldots,0),n,e)_2^*$-game.
\end{cor}
\begin{proof}
Let $u_s$ and $x_s$ be the position and state vectors,
respectively at the end of round $s$, of the
$((M,0,\ldots,0),n,e)_2^*$-game in which Paul always asks the
alternating question, and Carole always chooses
$u_{s+1}=\odd(u_s)$. By Theorem \ref{thm:CaroleStrat}, we need
only transform $\odd^s(u_0)$ into $f_s$, where $u_0$ is the
position vector corresponding to the initial state vector
$(M,0,\ldots,0)$.  By definition of $\odd(u)$ and of one step
of the liar machine, this is accomplished by observing that
$x_s(i) 
=f_s(-s+2i)$ for all $0\leq i\leq s$. Consequently
$\odd^n(u_0)(1)\leq e$ iff $\sum_{i=0}^ef_n(-n+2i)\geq 1$.
\end{proof}
The converse is not true.  For some games the alternating
question strategy is not optimal, so that Paul has a winning
strategy, but $\odd^n(u_0(1))>e$. For example, Paul can win the
$((1,11),4,1)^*_2$-game (as the reader can readily verify --
the first question is $(1,4)$), but the progression of
configurations given by the liar machine is $(1,11)\rightarrow
(0,7)\rightarrow (0,3)\rightarrow (0,1)\rightarrow (0,0)$.

We again use the following notation. Let $n \rightarrow
\infty$, fix $f \in (0,1/2)$, and set $ n_1 = n -
\floor{\frac{4}{(1 - 2f)^2} \log \log n}$ and $n_2 = n - n_1$.
Define $F=\lfloor f n\rfloor$, $F_1 = \lfloor f n_1\rfloor$,
and $F_2 = F-F_1$.

\begin{theorem}\label{ref:liarMachineBound}
Let $n,M\in \mathbb{Z}^+$.  Let
$f_0:\mathbb{Z}\rightarrow\mathbb{N}$ be the initial
configuration of the liar machine defined by $f_0(0)=M$, and
$f_0(j)=0$ otherwise.
For $n$ sufficiently large, if
$$
M \geq \frac{2^n}{\binom{n}{\leq F}}
    (2+o(1))c'\sqrt{n_2}
    ,
$$
where $c'$ is the constant from Theorem
\ref{thm:intervalupperbound}, then $\sum_{i=F_1}^{F} f_n(-n+2i)
\geq 1$.
\end{theorem}
\begin{proof}
Set $g_0=f_0$ and let $g_s$ be the chip distribution in the
linear machine after $s$ rounds.  Then for $F_1 \leq j \leq F$,
the number of chips at position $-n_1+2j$ in the linear machine
after $n_1$ rounds is
\begin{equation}\label{eqn:n1Linear}
g_{n_1}(-n_1+2j) = \frac{\binom{n_1}{j}}{2^{n_1}}
    \frac{2^{n}}{\binom{n}{\leq F}}(2+o(1))c'\sqrt{n_2}
    .
\end{equation}
Since $F<n_1/2$ for $n$ sufficiently large, the minimum occurs
at $j=F_1$, and is $\omega(\log n)$ by Lemma
\ref{lem:relativeCDF}.
Applying Theorem \ref{thm:ptwiseupperbound}, for $F_1\leq j\leq
F$ we have
\begin{equation}\label{eqn:n1Liar}
f_{n_1}(-n_1+2j) \geq
\frac{\binom{n_1}{j}}{2^{n_1}}
    \frac{2^{n}}{\binom{n}{\leq F}}(2+o(1))c'\sqrt{n_2}.
\end{equation}
Now for $F_1 \leq j \leq F$, define $h_{n_1}(-n_1+2j)$ to be
the right-hand side of \eqref{eqn:n1Liar}, and $h_{n_1}(j)=0$
elsewhere. Thus $h_{n_1}$ is obtained from $f_{n_1}$ by
removing chips outside of the interval $[-n_1+2F_1,-n_1+2F]$.
We run the linear machine with initial state $h_{n_1}$ for
$n_2$ rounds, and obtain for $F_1 \leq i \leq F$ that
$$
h_n(-n+2i) \geq
    \sum_{j=F_1}^{i}
    \frac{\binom{n_1}{j}}{2^{n_1}}
    \frac{2^{n}}{\binom{n}{\leq F}}(2+o(1))c'\sqrt{n_2}
     \frac{\binom{n_2}{i-j}}{2^{n_2}},
$$
as for $i$ and $j$ fixed, the contribution to $h_n(-n+2i)$ from
$h_{n_1}(-n+2j)$ is $h_{n_1}(-n+2j)\binom{n_2}{i-j}/2^{n_2}$.
Summing $h_n(-n+2i)$ over $i$ and applying Proposition
\ref{prop:HyperGeomApprox},
$$
\sum_{i=F_1}^{F}
    h_n(-n+2i)\geq c'\sqrt{n_2}(1+o(1)).
$$
Noting that $\sqrt{n_2} = o(F-F_1)$ and applying Theorem
\ref{thm:intervalupperbound} to $h_{n_1}$, we obtain
$\sum_{i=F_1}^{F} f_n(-n+2i) \geq 1 $ as desired.
\end{proof}

\begin{proof}[Proof of Theorem \ref{thm:pathGameBound}]
Corollary \ref{cor:liarGameToLiarMachine} reduces the
$((M,0,\ldots,0),n,e)_2^*$-game to the liar machine with
winning condition $\sum_{i=-n}^{-n+2e}f_n(i)\geq 1$, which
Theorem \ref{ref:liarMachineBound} shows is satisfied for the
given form of $M$.
\end{proof}

\section{Concluding remarks\label{sec:conclusion}}

The major open question is whether the time-evolution of the
liar machine with $M$ elements at the origin and zero elsewhere
can be given in closed form, or at least whether the leftmost
chip can be tracked more tightly.  Either case would yield an
improvement by decreasing the minimum $M$ for which Paul can
win the $((M,0,\ldots,0),n,e)^*_2$-game.  We suppose that the
best hope is for the optimal $M$ to be asymptotically a
constant multiple above the sphere bound.  Similarly, by the
reduction in \cite{SW92} from the $((M,0,\ldots,0),n,e)_2$-game
(original liar game) to the linear machine, improved tracking
of the leftmost chip could provide an alternative proof of
Theorem 3 of \cite{Z76}, which is equivalent to a lower bound
on $M$ for which Paul can win the original liar game.
Optimistically, the bound in \cite{Z76} on $M$ might be
improved to a constant multiple below the sphere bound.

We thank Joel Spencer for discussions that helped to
crystallize the ideas for this paper -- with the first author
during an extended collaboration on deterministic random walks,
and with the second author at a conference in 2004 on alternate
viewpoints for the liar game.

\end{document}